\documentclass[11pt]{amsart}

\usepackage{amsmath}
\usepackage{amssymb}
\usepackage{amsfonts}

\usepackage{amsmath}
\usepackage{amsfonts}
\usepackage{amssymb}
\usepackage{amsthm}
\usepackage{hyperref}
\usepackage{graphicx}
\usepackage{subcaption} 
\usepackage{xcolor}

\usepackage{euscript}
\usepackage{setspace}
\usepackage{tikz}
\usetikzlibrary{calc, intersections, through, backgrounds, positioning}

\usepackage{graphicx}

\newcommand{\C}{\mathbb{C}}
\newcommand{\h}{\mathcal{H}}
\newcommand{\diam}{{\operatorname{diam}}}
\newcommand{\dist}{{\operatorname{dist}}}
\newtheorem{definition}{Definition}
\newtheorem{theorem}{Theorem}

\newtheorem*{theoremA}{Theorem {$\mathbf{1'}$}}
\newtheorem{corollary}{Corollary}

\newtheorem{lemma}{Lemma}

\theoremstyle{remark}

\newtheorem{remark}{Remark}

\newcommand{\R}{\mathbb R}

\begin{document}
\title[Nodal sets of Dirichlet Laplace eigenfunctions]{The sharp upper bound for the area of the nodal sets of Dirichlet Laplace eigenfunctions}
\author{A. Logunov}
\address{A.L.:\,Department of Mathematics, Princeton University, Princeton,\,NJ, USA}
\email{log239@yandex.ru}
\author{ E. Malinnikova}
\address{E.M.:\,Department of Mathematics, Stanford University, Stanford,\,CA, USA\ \ \ \ \ \ \ 
Department of Mathematical Sciences, NTNU, Trondheim, Norway}
\email{eugeniam@stanford.edu}
\author{N. Nadirashvili}
\address{N.N:\,CNRS, Institut de Math\'ematiques de Marseille, Marseille, France}
\email{nikolay.nadirashvili@univ-amu.fr}
\author{ F. Nazarov}
\address{F.N.:\,Department of Mathematics, Kent State University, Kent,\,OH, USA}
\email{nazarov@math.kent.edu}

\begin{abstract} Let $\Omega$ be a bounded  domain in $\R^n$ with $C^{1}$ boundary and let $u_\lambda$ be a Dirichlet Laplace eigenfunction in $\Omega$ with eigenvalue $\lambda$.  We show that the $(n-1)$-dimensional Hausdorff measure of the zero set of $u_\lambda$ does not exceed $C(\Omega)\sqrt{\lambda}$. This result is new even for the case of domains with $C^\infty$-smooth boundary. 
	\end{abstract}
\maketitle

\section{Introduction} Let $\Delta_M$ be the Laplace operator on a compact $n$-dimensional Riemannian manifold and let $u_\lambda$ be an eigenfunction of $-\Delta_M$ with the eigenvalue $\lambda$, i.e., $\Delta_M u_\lambda+\lambda u_\lambda=0$. Denote by $Z(u_\lambda)=\{u_\lambda=0\}$ the zero set of $u_\lambda$.
S.~T.~Yau \cite{Y} conjectured that the surface area  of the zero set  of $u_\lambda$  satisfies the following inequalities
\[c\sqrt{\lambda}\le \h^{n-1}(Z(u_\lambda))\le C\sqrt{\lambda},\]
where the constants $c,C$ depend on $M$.
This conjecture was proved by Donnelly and Fefferman in \cite{DF} under the assumption that the metric is real analytic. 
The lower bound  and a polynomial in $\lambda$ upper bound  were obtained recently by the first author  in \cite{L2} and \cite{L1} respectively. 

In this article we consider the case of eigenfunctions of the Euclidean Laplace operator on a bounded domain with sufficiently regular boundary and the Dirichlet boundary condition. One of our results is the following.
\begin{theorem}\label{th:main}
 Let $\Omega$ be a bounded domain in $\R^n$ with $C^{1}$ boundary and let
 $u_\lambda$ be an  eigenfunction of the Laplace operator with the Dirichlet boundary condition, $\Delta u_\lambda+\lambda u_\lambda=0$ and $u_\lambda|_{\partial\Omega}=0$.   Then
\begin{equation}\label{eq:main} \h^{n-1}(Z(u_\lambda))\le C\sqrt{\lambda},\end{equation}
where $C$ depends only on  $\Omega$. 
\end{theorem}
The lower bound
\[\h^{n-1}(Z(u_\lambda))\ge c\sqrt{\lambda},\]
for sufficiently large $\lambda$, follows from the results of Donnelly and Fefferman in \cite{DF} combined with Lemma \ref{l:DF} below. We remark that this bound also holds for any solution of the equation $\Delta u_\lambda+ \lambda u_\lambda=0$ and the boundary condition plays no role. This follows from the fact that the zero set is $C\lambda^{-1/2}$ dense and a non-trivial result of \cite{L2}. The inequality \eqref{eq:main} was also proved by Donnelly and Fefferman  in \cite{DF1} for the case of real analytic boundary $\partial\Omega$. Their result was generalized to eigenfunctions of elliptic operators with real analytic coefficients by Kukavica \cite{K}. Similar estimates were recently obtained by Lin and Zhu \cite{LZ} for eigenfunctions of the bi-Laplace operator with various boundary conditions under the assumption that the boundary is real analytic. Also, the polynomial (in the eigenvalue) upper bounds for the area of the zero set of the Dirichlet, Neumann, and Robin eigenfunctions in smooth bounded domains in $\R^n$ were proved by Zhu in \cite{Z}.  

Our proof of Theorem \ref{th:main} is based on the results of Donnelly and Fefferman and the ideas developed in \cite{LM0, L1,L2}. In particular, we reduce the statement of the theorem to an estimate of the size of the nodal set of a harmonic function with controlled doubling index (the doubling index in defined in Section \ref{s:double} below). The novelty of the current work is the treatment of domains with non-analytic boundaries. More precisely, we work with Lipschitz domains in the Euclidean space and assume that (locally) the Lipschitz constant is small enough; the precise definition and the formulation of the main result are given in the next  section. This class of domains was recently considered by Tolsa \cite{T} in a different problem.  

The rest of the article is organized in the following way.  In Section \ref{s:double} we first discuss the doubling index of harmonic functions and its (weak) monotonicity properties near the boundary of Lipschitz domains with small Lipschitz constant, and  then we formulate the main estimate for the size of the zero set of harmonic functions in terms of the doubling index, see Theorem \ref{th:m2} below. Two auxiliary results are contained in Section \ref{s:two}, where the low regularity of the boundary requires some careful considerations. We prove Theorem \ref{th:m2} for harmonic functions in Section \ref{s:main}, and explain how Theorem \ref{th:main} follows from Theorem \ref{th:m2} in Section \ref{s:eigen}. 


{\bf{Acknowledgements.}}The authors are grateful to Misha Sodin for constant encouragement and motivation,  as well as for the organization of research visits and workshops that allowed the authors to meet and discuss the current work, including the workshop on nodal sets of eigenfunctions at IAS, Princeton, in February 2017. Our special thanks also go to the Tel Aviv University, where this work started.  A.L. was supported in part by the Packard
Fellowship and Sloan Fellowship. E.M. was partially supported by NSF grant
DMS-1956294 and by Research Council of Norway, Project 275113. F.N.
was partially supported by NSF grant DMS-1900008.


\section{Preliminaries}\label{s:pre}
\subsection{Smoothness of the boundary} Some of the tools used in the current paper should be compared to those in \cite{T}, where the following {\it boundary uniqueness conjecture} is studied.

{\it Let $h$ be a bounded harmonic function in a Lipschitz domain $\Omega$. Assume that $h$ vanishes on a relatively open set $U\subset\partial\Omega$  and $\nabla h$ vanishes on a subset of $U$ of positive surface measure. Then $h=0$.}

 Recently Xavier Tolsa  verified the conjecture   for Lipschitz domains with  small Lipschitz constant, see \cite{T}. We use the following definition.
 
\begin{definition}\label{d:Lip}
 Let $\Omega$ be a domain in $\R^d$, $\tau\in(0,1)$, and let $B=B(x,r)$ be a ball centered on $\partial\Omega$. We say that $\partial\Omega$ is $\tau$-Lipschitz  in $B$ if there is an isometry $T:\R^d\to\R^d$ and a function $f: B^{d-1}(0,r)\to \R$ such that $T(0)=x$, $f$ is a Lipschitz function with the Lipschitz constant bounded by $\tau$, $f(0)=0$, and
\[\Omega\cap B=T\left(\{(y', y'')\in B^{d}(0,r)\subset \R^{d-1}\times\R: y''>f(y')\}\right).\] 
In this case we write $\partial\Omega\cap B\in Lip(\tau)$. 
\end{definition}

\begin{remark}\label{r:1}
Our considerations are mostly local. When considering the part of the boundary $\partial\Omega\cap B\in Lip(\tau)$, we choose local coordinates in $B=B(x,r)$, $x\in\partial\Omega$, so that the isometry $T$ in the definition is the identity. We denote by $e_d$ the unit vector in the direction of the last coordinate, so that $x+\varepsilon e_d\in\Omega$ for $0<\varepsilon<r$. 
\end{remark}

\begin{remark}\label{r:2}
Note  that if $\partial\Omega\cap B\in Lip(\tau)$ and $B_1\subset B$ is a ball centered on $\partial\Omega$, then $\partial\Omega\cap B_1\in Lip(\tau)$. Also, rescaling does not change the Lipschitz constant. So if $\partial\Omega\cap B\in Lip(\tau)$ and $x\in\partial\Omega$ is the center of $B$, then, denoting  $\Omega_c=\{x+c(y-x): y\in\Omega\}$ and $B_c=cB=\{x+c(y-x): y\in B\}$ for some $c>0$, we have $\partial\Omega_c\cap B_c\in Lip(\tau)$.  
\end{remark}

\begin{definition}
We say that $\Omega$ is a Lipschitz domain with  local Lipschitz constant $\tau$ if there exists $r>0$ such that $\partial\Omega\cap B(x,r)\in Lip(\tau)$ for any $x\in\partial\Omega$.
\end{definition}
Clearly, any bounded $C^1$ domain is a domain with local Lipschitz constant $\tau$ for any positive $\tau$. So  Theorem \ref{th:main} follows from the next result.

\begin{theoremA}
 For each $n$, there exists $\tau_n>0$ such that the following statement holds. Let $\Omega$ be a bounded Lipschitz domain in $\R^n$ with local Lipschitz constant $\tau_n$ and let
 $u_\lambda$ be an  eigenfunction of the Laplace operator in $\Omega$ with the Dirichlet boundary condition, $\Delta u_\lambda+\lambda u_\lambda=0$ and $u_\lambda|_{\partial\Omega}=0$.   Then
\begin{equation*} \h^{n-1}(Z(u_\lambda))\le C\sqrt{\lambda},\end{equation*}
where $C$ depends only on  $\Omega$. 
\end{theoremA}

The constant $C$ depends only on the parameter $r$ for $\Omega$ in the definition of a Lipschitz domain with local Lipschitz constant $\tau_n$, on the diameter of $\Omega$, and on the dimension $n$.
In what follows we assume that the dimension of the ambient Euclidean space is fixed so usually we will not emphasize the dependence of our constants on it.

The rest of the article is devoted to a proof of Theorem ${1'}$.  We start with the following property of Lipschitz domains.
\begin{lemma}\label{l:star} Suppose that $\partial\Omega\cap B\in Lip(\tau)$, $
\tau<1/4$, where $B=B(x,r)$ and $x\in\partial\Omega$. We choose coordinates as in Remark \ref{r:1}. Let $x_0\in \overline{\Omega}\cap \frac{1}{4}B$, and let $x_1=x_0+\tau r e_d$. Then $\Omega\cap B(x_1, r/2)$ is star-shaped with respect to $x_1$.
\end{lemma}
A version of this lemma can be found in \cite{KN}. We provide a proof for the convenience of the reader.

\begin{proof} 
Let $x_1=(x_1',x_1'')$. Suppose that $x_2=(x_2',x_2'')\in \Omega\cap B(x_1,r/2)$.  Let now $x_3=(x_3',x_3'')$ be a point on the interval $(x_1,x_2)$. Clearly $x_3\in B(x_1, r/2)$. We want to check that $x_3''>f(x_3')$.

Let $x_3=ax_1+(1-a)x_2$, $a\in(0,1)$.  We have  $x_1''\ge f(x_1')+\tau r$ and $x_2''>f(x_2')$.  Therefore, we obtain
\[x_3''=ax_1''+(1-a)x_2''> af(x_1')+a\tau r+(1-a)f(x_2').\]
Then, since $f(x_1')\ge f(x_2')-\tau |x_1'-x_2'|>f(x_2')-\tau r/2$, we have
\[x_3''>f(x_2')+a\tau r-\frac{a\tau r}{2}=f(x_2')+\frac{a\tau r}{2}>f(x_3'),\]
where the last inequality holds since $|x_3'-x_2'|=a|x_1'-x_2'|<ar/2$ and $f$ is $\tau$-Lipschitz.
\end{proof}


\subsection{Some observations} 
In this section we recall some results about harmonic functions. 

Suppose that $h$ is a harmonic function in $\Omega$, $h\in C(\overline{\Omega})$, and $h=0$ on $\partial\Omega\cap B$, where $B=B(x,r)$ and $x\in\overline{\Omega}$. We define the function $v$ in $B$ by  $v=h^2$ in $\Omega\cap B$, and $v=0$ in $B\setminus \Omega$.  Then $v$ is  subharmonic in $B$ and the mean-value theorem implies that for any $y\in B(x,r/2)\cap\Omega$,
\begin{equation}\label{eq:subh}
h^2(y)\le \frac{1}{|B(y,r/2)|}\int_{B(y,r/2)\cap\Omega}h^2\le \frac{1}{|B(y,r/2)|}\int_{B(x,r)\cap\Omega}h^2,
\end{equation}
where $|E|$ is the $d$-dimensional Lebesgue measure of the set $E$.

Another known fact that we use is the following quantitative version of the Cau\-chy uniqueness theorem.  

\begin{lemma}\label{l:Cauchy}  Let $B_+$ be the  half-ball, 
\[B_+=\{(x',x'')\in\R^{d-1}\times\R: |x'|^2+(x'')^2< 1, x''> 0\}.\]
There exist  $\gamma\in(0,1)$ and $C>0$ such that if $h$ is harmonic in $B_+$, $h\in C^1(\overline{B}_+)$ and satisfies the inequalities $|h|\le 1, |\nabla h|\le 1$ in $B_+$ and $|h|\le\varepsilon$, $|\partial_d h|\le \varepsilon$ on $\Gamma=\{(x',x'')\in \overline{B}_+, x''=0\}$, $\varepsilon\le 1$, then 
\[|h(x)|\le C\varepsilon^\gamma\quad{\text{when}}\quad x\in\frac{1}{3}B_+=\left\{(x',x''):|x'|^2+(x'')^2< \frac1{9}, x''> 0\right\}.\] 
\end{lemma} 

 The reader can find  a proof of a similar statement in \cite{L} and a general result on second order elliptic PDEs in Lipschitz domains in \cite{ARRV}.  A simple proof  is also given in Section \ref{ss:cau} for the convenience of the reader.

\section{The doubling index}\label{s:double}
\subsection{The doubling index inside the domain} 
Let $h\in C(\overline{\Omega})$ be a non-zero harmonic function in a domain $\Omega\subset\R^d$. For each $x\in\overline{\Omega}$ and  $r>0$, we define \begin{equation}\label{eq:double}
H_h(x,r)=\int_{B(x,r)\cap\Omega}h^2\quad {\text{and}}\quad N_h(x,r)=\log\frac{H_h(x,2r)}{H_h(x,r)},
\end{equation}
and, with some abuse of language, we call $N_h(x,r)$ the doubling index of $h$ in $B=B(x,r)$. 

Assume first that $B(x,2R)\subset\Omega$, then 
\begin{equation}\label{eq:monh}
N_h(x,r)\le N_h(x,R),\quad{\text{when}}\quad r<R.
\end{equation} An elementary proof can be obtained by  decomposing $h$ into spherical harmonics, see, e.g., \cite{KM}.
This is a simple and useful result, its various versions go back to the works of Landis \cite{La}, Agmon \cite{Ag}, and Almgren \cite{Al}. 

Suppose that  $B(x,4r)\subset\Omega$. Then we rewrite the inequality $N_h(x,r)\le N_h(x,2r)$ as
\begin{equation}\label{eq:three2}
\left(\int_{B(x,2r)}h^2\right)^2\le \int_{B(x,r)}h^2\int_{B(x,4r)}h^2.
\end{equation}
Similarly to \eqref{eq:subh}, for any $y\in B(x, 3r/2)$,   we have
\[h^2(y)\le \frac{1}{|B(y,r/2)|}\int_{B(y,r/2)}h^2\le \frac{1}{|B(y,r/2)|}\int_{B(x,2r)}h^2.\]
Finally, applying \eqref{eq:three2} and  using the trivial bound of the $L^2$ norms by the $L^\infty$ norms, we obtain
\begin{equation}\label{eq:threeh}
\sup_{B(x,3r/2)}|h|\le 2^d(\sup_{B(x,r)}|h|)^{1/2}(\sup_{B(x,4r)}|h|)^{1/2}.
\end{equation}

\subsection{The doubling index on the boundary}
We need a version of the monotonicity formula \eqref{eq:monh} and the three ball inequality \eqref{eq:threeh} near a part of the boundary on which the harmonic  function vanishes.  
First, we recall a lemma that is  proven in \cite{KN}. 
\begin{lemma}[Kukavica, Nystr\"om]\label{l:starm} Let $\Omega$ be a domain in $\R^d$ and let $B_1$ be a ball centered on $\partial\Omega$ such that $\partial\Omega\cap B_1$ is $C^3$ smooth. Let also  $x\in\Omega$ be such that $\Omega\cap B(x,R)$ is star-shaped with respect to $x$, $B(x,R)\subset B_1$. Suppose that $h\in C(\overline{\Omega})$ is a non-zero harmonic function in $\Omega$ and $h=0$ on $\partial\Omega\cap B_1$.  Then
\begin{equation}\label{eq:KN}
 \log\frac{H_h(x, r_2)}{H_h(x,r_1)}\le \frac{\log (r_2/ r_1)}{\log (r_3/r_2)}\log \frac{H_h(x, r_3)}{H_h(x,r_2)},
 \end{equation}
when $0<r_1<r_2<r_3<R$. 
\end{lemma}
The assumption that the boundary of $\Omega$ is $C^3$ smooth implies that  $h\in C^2(\overline{\Omega}\cap B_1)$, so every integration by parts in \cite{KN} can be easily justified.

Now we prove the following almost monotonicity property of  the doubling index in Lipschitz domains. 

\begin{lemma}\label{l:mon} Let $\Omega$ be a domain in $\R^d$.   For any $\varepsilon>0$, there exists $\tau_\varepsilon>0$ such that if $\tau<\tau_\varepsilon$, $\partial\Omega\cap B\in Lip(\tau)$, where $B=B(x,R),\ x\in\partial\Omega$, and $h\in C(\overline{\Omega})$  is a non-zero harmonic function in $\Omega$, $h=0$ on $\partial\Omega\cap B$, then
 \begin{equation}\label{eq:monLip}
 N_h(x_0,r)\le (1+\varepsilon) N_h(x_0,2r),
 \end{equation}
 for any  $x_0\in \overline{\Omega}\cap \frac{1}{4}B$ and $r<R/16$.
 \end{lemma} 
 
 We remark that a stronger result holds when the  boundary of the domain is smooth. For example for the case of a $C^{1,Dini}$ domain the inequality \eqref{eq:monLip} can be replaced by $N_h(x_0,r_1)\le (1+\varepsilon)N(x_0, r_2)$ when $r_1<r_2<R/8$, see \cite{KN}\footnote{We refer the reader also to the preceding works \cite{AEK} and \cite{AE} for related results.}.
 Thus for $C^{1,Dini}$ domains, we know that the doubling index over balls centered on $\partial\Omega\cap B$ stays uniformly bounded. We do not know if this is still true for Lipschitz domains.
 For the case of the domains with small Lipschitz constant, we can conclude only that the doubling index $N_h(x_0,r)$ 
 does not grow faster than $r^{-a}$ with some small positive $a$ as $r\to 0$, which is sufficient for our purposes.

 \begin{proof}  First we assume that $\partial\Omega\cap B$ is a graph of a $C^3$-smooth function.   Let $e_d$ be as in Remark \ref{r:1} and let $x_1=x_0+16\tau re_d$. We assume that $\tau<1/16$. Then by Lemma \ref{l:star} we see that $B(x_1,8r)\cap \Omega$ is star-shaped with respect to $x_1$. 
 We apply \eqref{eq:KN} and  obtain
 \begin{multline*}N_h(x_0,r)=\log\frac{H_h(x_0,2r)}{H_h(x_0,r)}\le \log\frac{H_h(x_1, (2+16\tau)r)}{H_h(x_1, (1-16\tau)r)}\\
 \le \frac{\log((2+16\tau)/(1-16\tau))}{\log((4-16\tau)/(2+16\tau))}\log\frac{H_h(x_1, (4-16\tau)r)}{H_h(x_1, (2+16\tau)r)}\\
 \le (1+O(\tau))\log\frac{H_h(x_0, 4r)}{H_h(x_0, 2r)}=(1+\varepsilon) N_h(x_0, 2r),
 \end{multline*}
 when $\tau$ is small enough.
 
 We want to drop the assumption that $\partial\Omega\cap B$ is $C^3$ smooth.  We fix the ball $B$ and assume that $\partial\Omega\cap B$ is given by the graph of a Lipschitz function $f: B^{d-1}(0,R)\to\R$ with the Lipschitz constant bounded by $\tau$. In this coordinate system the ball $B$ is identified with $B^d(0,R)$ 
 
 Let $\varphi$ be a mollifier supported in the unit ball of $\R^{d-1}$ and let, as usual,  $\varphi_\delta(x)=\delta^{-(d-1)}\varphi\left(\frac{x}{\delta}\right)$.
 We define  $f_n=f*\varphi_{R/n}+\tau R/n$.  Then $\{f_n\}$  is a sequence of $C^3$ smooth functions such that 
 \[f_n:B^{d-1}(0,(1-1/n)R)\to \R,\quad f(y')<f_n(y')<f(y')+2\tau R/n,\]
  and the Lipschitz constant of $f_n$ is also bounded by  $\tau$. We also define
 \[\Omega_n=\{y=(y',y'')\in B^d(0, (1-1/n)R)\subset\R^{d-1}\times\R: y''>f_n(y')\}.\]
 Clearly, $\Omega_n\subset B^d(0, (1-1/n)R)\cap \Omega$. 
 Let
 \[\Gamma_n=\{y=(y',y'')\in B^d(0,(1-1/n)R): y''=f_n(y')\}.\]
 First, we see that $\delta_n=\sup_{\Gamma_n}|h|$ converge to zero as $n\to\infty$ since $h$ is uniformly continuous on $\overline{\Omega}\cap\overline{B}$, $h=0$ on $\partial\Omega$, and $\dist(y,\partial\Omega)<2\tau R/n$ when $y\in\Gamma_n$.

 Next, we consider the harmonic function $h_n$ in $\Omega_n$ such that on $\partial\Omega_n$
 \[h_n(x)=\begin{cases}h(x)-\delta_n,\ {\text{if}}\ h(x)>\delta_n,\\
 0,\quad\quad\quad\ \ \ {\text{if}}\ |h(x)|\le\delta_n,\\
 h(x)+\delta_n,\ {\text{if}}\ h(x)<-\delta_n.
 \end{cases}\]
 Clearly we have $h_n\in C(\overline{\Omega}_n)$, $h_n=0$ on $\Gamma_n$, and, by the maximum principle,  $|h-h_n|\le\delta_n$ in $\Omega_n$.
  Thus $h_n\rightarrow h$ uniformly on compact subsets of $B\cap\Omega$. 
  
  We fix $x_0\in\Omega\cap\frac14B$ and $r\in(0,R/16)$. Then $x_0\in \Omega_n\cap B(0,(1-1/n)R/4)$ and $r<(1-1/n)R/16$ for $n$ large enough. Also, $|h_n|\le \max_{\overline{\Omega}\cap\overline{B}}|h|$ and $|(\Omega\cap B(x,R))\setminus \Omega_n|\to 0$ as $n\to\infty$. Then we have
  \[N_h(x_0,r)=\frac{\int_{B(x_0,2r)\cap\Omega}h^2}{\int_{B(x,r)\cap\Omega}h^2}=\lim_{n\to\infty}\frac{\int_{B(x_0,2r)\cap\Omega_n}h_n^2}{\int_{B(x_0,r)\cap\Omega_n}h_n^2}=\lim_{n\to\infty}N_{h_n}(x_0,r).\]
  The inequality \eqref{eq:monLip} is now obtained as the limit of  the corresponding inequalities for $h_n$. Finally the required inequality \eqref{eq:monLip} for $x_0\in \partial\Omega\cap\frac14B$  follows by taking the limit as $\varepsilon\to 0+$ of the corresponding inequalities for $x_0+\varepsilon e_d$.
 \end{proof}
 
 \begin{corollary}\label{cor:x12} Let $\partial\Omega\cap B\in Lip(\tau)$, $\tau<\tau_\varepsilon$, $B=B(x,R)$, $x\in\partial\Omega$, and let  $x_1,x_2\in \overline{\Omega}\cap \frac{1}{4}B$  with $|x_1-x_2|<r/4$ and $r<R/8$. If $h\in C(\overline{\Omega})$ is a non-zero harmonic function in $\Omega$ such that $h=0$ on $\partial\Omega\cap B$,  then
 \[N_h(x_1,r/2)\le  3(1+\varepsilon)^2 N_h(x_2, r).\]
 \end{corollary} 

\begin{proof} Note that $B(x_1,r)\subset B(x_2, 2r)$ and $B(x_1, r/2)\supset B(x_2, r/4)$. Thus we obtain
\[N_h(x_1, r/2)\le \log\frac{\int_{B(x_2,2r)\cap\Omega}h^2}{\int_{B(x_2, r/4)\cap \Omega}h^2}=N_h(x_2, r/4)+ N_h(x_2, r/2)+N_h(x_2, r).\]
Now Lemma \ref{l:mon} implies the required estimate.
\end{proof}


\subsection{Three ball inequality}
We apply the monotonicity lemma a number of times. First, we claim that it implies
 a version of the three ball theorem for harmonic functions vanishing on some part of the boundary. 

\begin{lemma}\label{l:three1} Let $\Omega$ be a  domain in $\R^d$ and let $B$ be a ball centered on $\partial\Omega$. We assume that $\partial\Omega\cap B\in Lip(\tau)$, where $\tau$ is small enough. Then for any function $h\in C(\overline{\Omega})$ harmonic in $\Omega$ and vanishing on $\partial\Omega\cap B$, we have
\[\sup_{ \frac{3}{2}B_0\cap\Omega}|h|\le 3^d(\sup_{B_0\cap\Omega}|h|)^{1/3} (\sup_{4B_0\cap\Omega}|h|)^{2/3}, \]
 for any  ball $B_0$ with the center in $\overline{\Omega}\cap \frac{1}{4}B$ and such that $16B_0\subset B$.
\end{lemma}

\begin{proof}  Assume that $\tau<\tau_1$ given by Lemma \ref{l:mon} for the case $\varepsilon=1$. Let  $B_0=B(x_0,r)$. We apply  Lemma \ref{l:mon}. Taking the exponentials,  we obtain
\[
\int_{2B_0\cap\Omega} h^2\le \left(\int_{B_0\cap\Omega}h^2\right)^{1/3}\left(\int_{4B_0\cap\Omega}h^2\right)^{2/3}.
\]
Then \eqref{eq:subh} and the trivial bound of the $L^2$-norm by the $L^\infty$-norm  imply that for any $y\in \frac{3}{2}B_0\cap\Omega$,
\begin{multline*}\label{eq:supth}
h^2(y)\le \frac{1}{|B(y,r/2)|}\int_{2B_0\cap\Omega}h^2
\le 8^d \left(\sup_{B_0\cap\Omega}|h|\right)^{2/3}\left(\sup_{4B_0\cap\Omega}|h|\right)^{4/3}.
\end{multline*}
\end{proof}


\subsection{The  maximal doubling index} 
Let $\Omega$ be a  domain in $\R^d$ and let  $\partial\Omega\cap B\in Lip(\tau)$ where $B$ is centered on $\partial\Omega$.  We consider a closed cube  $Q\subset \frac{1}{32}B$ such that $Q\cap \Omega\neq\varnothing$.
Assume that a non-zero function $h\in C(\overline{\Omega})$ is harmonic in $\Omega$ and vanishes  on $\partial\Omega\cap B$ and let  $\ell=\diam(Q)$. We define the  maximal doubling index of  $h$ in $Q$    by 
\begin{equation}\label{eq:N*}
N^*_h(Q)=\sup_{x\in Q\cap\overline{\Omega}, \frac{\ell}{2}\le r\le \ell}  N_h(x,r).
\end{equation}
 Clearly the function $(x,r)\mapsto N_h(x,r)$ is continuous on $(Q\cap\overline{\Omega})\times[\ell/2,\ell]$. Therefore the supremum above is finite.

 Lemma \ref{l:mon} on the monotonicity of the doubling index implies that  if $\varepsilon>0$ and $\tau<\tau_\varepsilon$, then
for any cube $Q_1\subset Q\subset \frac{1}{32}B$ and $Q_1\cap\Omega\neq\varnothing$, we have 
 \[N_h^*(Q_1)\le \left(\frac{2s(Q)}{s(Q_1)}\right)^{2\varepsilon} N_h^*(Q),\] 
 where $s(Q)$ is the side length of the cube $Q$; we have used the inequality $\log_2(1+\varepsilon)\le 2\varepsilon$.

 \subsection{A version of the main result for harmonic functions}
Let $\Omega$ be a domain in $\R^d$ and let $h\in C(\overline{\Omega})$ be a non-zero harmonic function in $\Omega$. We assume that $h=0$ on the part $\partial\Omega\cap B$ of the boundary, where $B$ is a ball centered on $\partial\Omega$ and $\partial\Omega\cap B\in Lip(\tau)$.  Our aim is to estimate the $(d-1)$-dimensional measure of the zero set of $h$  using the  doubling index of $h$. We define the zero set of $h$ by 
\[Z(h)=\{x\in\Omega: h(x)=0\},\] so that the boundary points are not included into the zero set.

\begin{theorem}\label{th:m2}  Let  $\Omega\subset \R^{d}$, let $x\in\partial\Omega$ and let $r>0$ be such that $\partial\Omega\cap B(x,128r)\in Lip(\tau),$ where $\tau$ is small enough. Then there exists $C$ such that  
\[\h^{d-1}(Z(h)\cap B(x,r))\le C(N_h(x,4r)+1)r^{d-1},\]
for any non-zero function $h\in C(\overline{\Omega})$ that is harmonic in $\Omega$ and satisfies $h=0$ on $\partial\Omega\cap B(x,128x)$.
\end{theorem}
Theorem \ref{th:m2} is proved in  Section \ref{ss:ptm2}. We then deduce  Theorem $\ref{th:main}' $ in Section \ref{ss:thm}, where we consider the harmonic extension of the eigenfunction and use Lemma \ref{l:DF} below to estimate the doubling index of the extension by a multiple of the square root of the eigenvalue. 

Theorem \ref{th:m2}  allows us to estimate the area of the zero set of a harmonic function near the part of the boundary  where the function vanishes. 
We remark also that the  estimate for the zero set inside the domain was proved by Donnelly and   Fefferman in \cite{DF}.  

\begin{lemma}[Donnelly, Fefferman]\label{l:in} 
Let   $h$ be a non-zero harmonic function in $\Omega\subset \R^{d}$.
There exists $C$  such that 
\[\h^{d-1}(Z(h)\cap B)\le C (N_h(x,4r)+1)r^{d-1},\]
for any ball $B=B(x,r)$  that satisfies $\overline{B}(x,8r) \subset\Omega$. 
\end{lemma}

The proof follows from the argument in \cite{DF}, some versions of this result can be also found in \cite{L} and \cite{H}. We outline some steps of the  proof for the interested reader in the Appendix, see \ref{ss:DF}.

 \section{Two auxiliary lemmas}\label{s:two}
  \subsection{A standard construction}
  In this section we give two versions of the Hyperplane Lemma. We suggest  that the reader compares the statements to the one of \cite[Lemma 4.1]{L1}. Both statements refer to the following construction.
 
\begin{figure}
	\centering
	\includegraphics[scale=0.7]{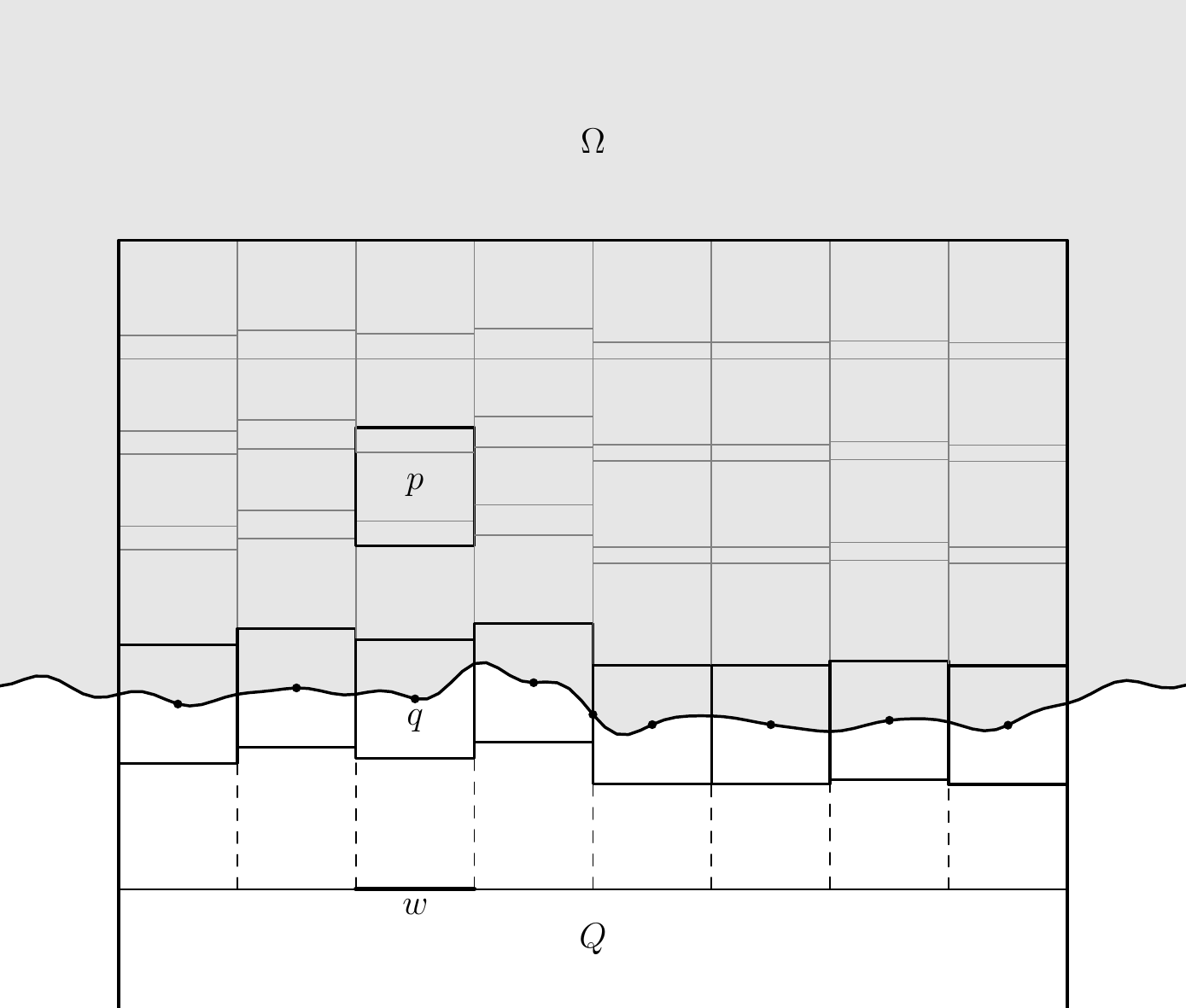}
	\caption{The standard construction}\label{f:1}
\end{figure}

 Assume that $\Omega\subset\R^{d}$  and $\partial\Omega\cap B\in Lip(\tau)$, where $B$ is a ball centered on $\partial\Omega$ and $\tau\in(0, (16\sqrt{d})^{-1})$. We fix a coordinate system as in Remark \ref{r:1}. Let  $Q$ be a cube centered at $x_Q=(x_Q',x_Q'')\in\partial\Omega\cap B$ whose sides are parallel to the axes of this coordinate system   and such that $Q\subset B$.  As above, the  side length of $Q$ is denoted by $s(Q)$. Our choice of  $\tau$ implies that $\partial\Omega$ does not intersect the two faces of the cube $Q$ which are orthogonal to $e_d$, moreover, $\partial\Omega\cap Q$ is contained in the middle part $\{(x',x'')\in Q: |x''-x_Q''|<s(Q)/4\}$ of $Q$. 
 
 Let $k\ge 3$. We partition the projection $\pi(Q)$ of $Q$ to the hyperplane $\R^{d-1}\times\{0\}$ into $2^{k(d-1)}$ small equal cubes $w$ with the side length $s(w)=2^{-k}s(Q)$ in the usual way so that any two distinct small cubes  have no common inner points. For each small cube $w$, there is a uniquely defined $d$-dimensional cube $q$ such that $\pi(q)=w$ and the center of $q$ lies on $\partial\Omega\cap Q$.  Furthermore, we cover $(\pi^{-1}(w)\cap(\Omega\cap Q))\setminus{q}$ by at most $2^k$ cubes $p$ such that $p\subset Q$, $p$, $\pi(p)=w$, $p$ has no common inner points with $q$, and $s(p)=s(q)=2^{-k}s(Q)$, cubes $p$ may overlap.  See Figure \ref{f:1}.

 We denote the set of all boundary cubes $q$ by $\mathcal{B}_k(Q)$ and the set of all inner cubes $p$ by $\mathcal{I}_k(Q)$. Note that for each $p\in\mathcal{I}_k(Q)$, we have $\dist(p,\partial\Omega)>c s(p)$ for some absolute constant $c$. We call the triple $(Q,\mathcal{B}_k(Q), \mathcal{I}_k(Q))$ the standard construction. After we fix a  coordinate system, our standard construction depends on the choice of the cube $Q$ and the parameter $k$, the family $\mathcal{B}_k(Q)$ of the boundary cubes is defined uniquely and we may fix some choice for the inner cubes $\mathcal{I}_k(Q)$.  
  
 \subsection{The first hyperplane lemma}
In the first lemma we assume that the maximal doubling index $N_h^*(Q)$ is large enough.
\begin{lemma}\label{l:hp} 
There exist constants $k_0\ge 3$ and $N_0\ge 1$ such that for any integer $k\ge k_0,$ there exists $\tau(k)>0$ for which the following statement holds. Suppose that $\Omega$ is a  domain in $\R^d$, $\partial\Omega\cap B\in Lip(\tau)$, $\tau<\tau(k),$ and 
  $Q\subset \frac{1}{64}B$ is a cube as above centered on $\partial\Omega$.   Then for any function  $h\in C(\overline{\Omega})$  harmonic in $\Omega$, with  $h=0$ on $B\cap\partial\Omega$, and  $N^*_h(Q)>N_0$, 
   there exists  a cube $q\in \mathcal{B}_k(Q)$  such that $N^*_h(q)\le N_h^*(Q)/2$.
\end{lemma}

\begin{proof} 
	Let $x_Q$ be the center of the cube $Q$ and let $B_1=B(x_Q,\ell)$, where $\ell=\diam(Q)$. We have  $B_1\subset B$ and define  $M^2=\int_{B_1\cap\Omega}h^2$.
	
	Denote  $N=N_h^*(Q)$ and suppose  that  the inequality  $N^*_h(q)>N/2$ holds for each cube $q\in\mathcal{B}_k(Q)$. 
	Then for each such $q$, there exist $y_q\in q\cap\overline{\Omega}$ and $r_q\in(2^{-k-1}\ell, 2^{-k}\ell)$ such that $N_h(y_q,r_q)>N/2$. Suppose that 
	\begin{equation}\label{eq:tauepsilon}
	\tau<\tau_\varepsilon,\quad{\text{and}}\quad (1+\varepsilon)^k<2,
	\end{equation}
	where we use the notation of Lemma \ref{l:mon}.
Then the almost monotonicity of the doubling index, Lemma \ref{l:mon}, implies
	$N_h(y_q, 2^mr_q)>N/4$ when $0\le m\le k$.
		
	Assuming that $k\ge 20$, we apply  the estimate of  the doubling index $k-4$ times and use that $B(y_q, \ell/2)\subset B_1$ to obtain
	\begin{multline*}
		\int_{B(y_q, 2^{-k+2}\ell)\cap\Omega} h^2\le \int_{B(y_q, 8r_q)\cap\Omega} h^2\le e^{-N(k-4)/4}\int_{B(y_q,2^{k-1}r_q)\cap\Omega}h^2\\ \le e^{-N(k-4)/4}\int_{B(y_q,\ell/2)\cap\Omega}h^2
		\le e^{-Nk/5}M^2.
		\end{multline*}
	Next, we note that the integral estimate above implies a pointwise estimate in a smaller ball by \eqref{eq:subh}. We have 
	\begin{equation}\label{eq:L11}
	\sup_{B(y_q,2^{-k+1}\ell)\cap \Omega}h^2\le C2^{dk}\ell^{-d}\int_{B(y_q,2^{-k+2}\ell)\cap\Omega}h^2\le C2^{dk}\ell^{-d}e^{-Nk/5}M^2,
	\end{equation}
where $C=C(d)$.
	
	As above, we  assume also that $\tau< (16\sqrt{d})^{-1}$.
	 For each cube $q\in\mathcal{B}_k(Q)$,  
	 denote by $q^+$ its upper quarter, where "up" is in the direction of  $e_d$.  Then $q^+\subset\Omega$ and $\dist(q^+, \partial\Omega)\ge 2^{-k}s(Q)/10$.  For $y\in q^+$, the standard Cauchy estimate implies 
	\[|\nabla h(y)|\le C2^k\ell^{-1}\sup_{B(y, 2^{-k}s(Q)/10)}|h|.\]
	We note that $B(y, 2^{-k}s(Q)/10)\subset B(y_q, 2^{-k+1}\ell)\cap\Omega$. Then combining the above inequality with \eqref{eq:L11}, we obtain
	\begin{equation}\label{eq:L12}
	\sup_{q^+}|\nabla h|\le  C 2^k\ell^{-1}\sup_{B(y_q, 2^{-k+1}\ell)\cap\Omega}|h|\le C2^{k(d+2)/2}\ell^{-(d+2)/2}e^{-Nk/10}M.
	\end{equation}

	Let $B_0=B(x_Q+3\cdot2^{-k-3}s(Q)e_d, s(Q)/2)$ 
	and let 
	\[B_{0,+}=\{x=(x',x'')\in B_0: x''\ge x''_Q+3\cdot 2^{-k-3}s(Q)\}\] be the upper half of $B_0$.
	We denote by  $\Gamma_{0}$ the flat part of the boundary of $B_{0,+}$. We note that $2B_0\subset B_1$. Assuming that  $\tau<2^{-k-3}$, we have $\dist(B_{0,+},\partial\Omega)\ge 2^{-k-2}s(Q)$. 
	Then using  \eqref{eq:subh} and the Cauchy estimate, we get
	\[\sup_{B_0\cap\Omega}|h|\le C\ell^{-d/2}M,\quad \sup_{B_{0,+}}|\nabla h|\le C2^k \ell^{-d/2-1}M. \]
	Also, by \eqref{eq:L11} and \eqref{eq:L12}, we have
	\[\sup_{\Gamma_0}|h|\le C2^{kd/2}\ell^{-d/2}e^{-Nk/10}M,\quad \sup_{\Gamma_{0}}|\nabla h|\le C2^{kd/2+k}\ell^{-d/2-1}e^{-Nk/10}M,\]
	since $\Gamma_0\subset\bigcup_{q\in \mathcal{B}_k(Q)}q^+$.
	
  Applying 
	Lemma \ref{l:Cauchy}  to $B_{0,+}$,  we get 
	\[\sup_{\frac13 B_{0,+}}|h|\le C 2^{\gamma kd/2+k}\ell^{-d/2}e^{-\gamma Nk/10}M.\]
	Let $y_Q=x_Q+s(Q) e_d/12$ and let $m$ be the least integer such that $2^m>16\sqrt{d}$.  Then $B_2=B(y_Q, 2^{-m}\ell)\subset \frac 13 B_{0,+}$ when $k$ is large enough (we remark that $B_{0,+}$ depends on $k$). Integrating the last inequality over $B_2$ and using that $vol(B_2)\le C\ell^d$, we obtain
	\[\int_{B_2} h^2\le C2^{\gamma kd+2k}e^{-\gamma Nk/5}M^2.\]
	
	Finally, we compare the last integral to  $\int_{B_1\cap \Omega} h^2=M^2$.   Note that $B_1\subset B(y_Q, 2\ell)=2^{m+1}B_2$.  By the almost monotonicity of the doubling index, recalling that $\tau<\tau_\varepsilon<\tau_1$,  we have 
		\begin{multline*}
		2^{m+1}N_h(y_Q, \ell)\ge \sum_{j=0}^m N_h(y_Q, 2^{-j}\ell)= \log \frac{\int_{B(y_Q,2\ell)\cap\Omega}h^2}{\int_{B_2}h^2}\\
		\ge \log \frac{\int_{B_1\cap\Omega}h^2}{\int_{B_2}h^2}\ge \gamma N k/5-\gamma kd-2k-C.
		\end{multline*}
	Since $N_h(y_Q,\ell)\le N^*_h(Q)=N$, we get $2^{m+1}N\ge \gamma Nk/5-\gamma kd-2k-C$.	
	Taking $k$ large enough we may achieve  $\gamma k/5>2^{m+2}$. Then the inequality above implies
	\[N\le \frac{10(\gamma kd+2k+C)}{\gamma k}\le 10\left(d+(2+C)\gamma^{-1}\right).\]
	 Taking $N_0=10\left(d+(2+C)\gamma^{-1}\right)$, we obtain a contradiction  for  $N>N_0$.  We also choose $\varepsilon=\varepsilon(k)$ such that $(1+\varepsilon)^k<2$  and finally choose $\tau(k)=\min\{\tau_\varepsilon,2^{-k-3}, (16\sqrt{d})^{-1}\}$.
	
	\end{proof}

\subsection{The second hyperplane lemma: cubes without zeros}
For cubes with the maximal doubling index bounded by $N_0$, we use the following version of the above statement.  The reader may compare it to Corollary 3.4.4 in \cite{LM}.

\begin{lemma}\label{l:hp1} 
 For any $N>0$ there exist $\tau(N)$ and $k(N)$  such that the following statement holds.
 Suppose that  $\Omega$ is a domain in $\R^d$, $\partial\Omega\cap B\in Lip(\tau)$,  $\tau<\tau(N)$, and 
$Q\subset \frac{1}{64}B$ is a cube centered on $\partial\Omega$. Let also $h\in C(\overline{\Omega})$ be  a non-zero function harmonic in $\Omega$, with  $h=0$ on $B\cap\partial\Omega$ and  $N^*_h(Q)\le N$. 
Then for any $k\ge k(N)$, there exists $q\in \mathcal{B}_{k}(Q)$  such that $Z(h)\cap q=\varnothing$.
\end{lemma}

We remark that in this version both $\tau$ and  $k$ depend on $N$. First, we prove the following version of the lemma for a half ball.

\begin{lemma}\label{l:toyl}
Let $B$ be the unit ball in $\R^d$ and let $B_+$ be the half ball,
\[B_+=\{y=(y',y'')\in\R^{d-1}\times\R: |y'|^2+y''^2<1, y''>0\}.\]
 Let $g$ be a function harmonic in $B_+$, $g\in C(\overline{B}_+)$,  $g=0$ on $\overline{B}_+\cap \{y''=0\}$, and 
 \[\sup_{\frac{1}{4} B_+}|g|=1.\]
 For any $N>0$, there exist $\rho=\rho(N)\in(0,1/16)$ and $c_0=c_0(N)>0$ such that if $N_g(0, 1/4)\le N$, then there is $x'\in\R^{d-1}$ with $|x'|<1/16$ such that 
 \[|g(y)|\ge c_0y'',\quad {\text{for any}}\quad  y=(y', y'')\in B((x',0), \rho)\cap B_+.\]
\end{lemma}

\begin{proof} 
Let $B_-$ be the reflexion of the half-ball $B_+$ with respect to the hyperplane $y''=0$. Then $g$ can be extended to a harmonic function in  $B$  by $g(y',y'')=-g(y',-y'')$ when $(y',y'')\in B_-$. We denote this extension by $g$ as well. 
The normalization $\sup_{\frac{1}{4} B_+}|g|=1$ and the standard Cauchy estimate imply that every partial derivative of $g$ is uniformly bounded in $B(0, 1/8)$. 

Let $\delta=\max_{x'\in\R^{d-1},|x'|\le 1/16}|\nabla g(x',0)|$. Lemma \ref{l:Cauchy}, applied to the half ball $\frac{1}{16}B_+$ implies that
\[\sup_{B(0,\frac1{64})}|g|\le C\delta^\gamma.\]
Then $\int_{B(0, \frac{1}{64})}g^2\le C\delta^{2\gamma}$ and $\int_{B(0,\frac12)}g^2\ge c\sup_{B(0,\frac14)}g^2=c$.  On the other hand, 
\[\log\frac{\int_{B(0,\frac12)}g^2}{\int_{B(0,\frac{1}{64})} g^2}
\le 5N_g(0,\frac14)\le 5N.\]
We have used that the doubling index of $g$ in $B_+$ and of the extension are the same for balls centered at the origin and that the doubling index inside the domain is monotone by \eqref{eq:monh}.
We conclude that $\delta\ge ce^{-3N\gamma^{-1}}$.

Let $x'_*\in \R^{d-1}$, $|x'_*|\le 1/16$, be such that $|\nabla g(x'_*,0)|=\delta$. Clearly we have $|\nabla g(x'_*,0)|=|\partial_d g(x'_*,0)|$ and we may assume that $\partial_d g(x_*',0)=\delta$, otherwise we consider the function $-g$. Then  $\partial_d g(x)>\delta/2$ when $\dist(x,(x'_*,0))<\rho=\min\{c_0\delta, 1/16\}$, where $c_0$ depends on the constant upper bound for the second derivatives of $g$ in $B(0,1/8)$. Therefore 
\[g(y)\ge \delta y''/2\ge ce^{-3N\gamma^{-1}}y'',\] when $y=(y',y'')\in B((x_*',0), \rho)$.  
\end{proof}

\begin{proof}[Proof of Lemma \ref{l:hp1}] Now we deduce Lemma \ref{l:hp1} from Lemma \ref{l:toyl}. By rescaling, see Remark \ref{r:2}, we can achieve that $s(Q)=4$. 
We may also assume  that  
\begin{equation}\label{eq:hnorm}
\sup_{B(x_Q, 3)\cap\Omega}|h|=1.
\end{equation} 

Let $x_1=x_Q-3\tau e_d$, $B_1=B(x_1,1)$, and let $B_{1,+}$ be the upper half of $B_1$. Let also $B_2=2B_1$. First, we consider the harmonic function $g_0$ such that $g_0=1$ on the upper half of the sphere $\partial B_2$ and $g_0=-1$ on the lower half of $\partial B_2$. We denote as usual $x_1''=x_1\cdot e_d$.  Clearly $g_0=0$ on 
\[\Gamma_0=\{x=(x',x'')\in B_2: x''=x_1''\}\]
 and $g_0\ge 0$ on $B_{2,+}$. We note that $\Gamma_0$ does not intersect $\overline{\Omega}$. Then $|h|\le g_0$ on $\Omega\cap B_2\subset B_{2,+}$ by the maximum principle. We also have $g_0(x)\le C_1 (x''-x''_1)$ when $x=(x',x'')\in B_{1,+}$, since $g_0=0$ on $\Gamma_0$ and $g_0$ has bounded derivatives in $B_1$. Therefore $|h(x)|\le C_1(x''-x''_1)$ when $x=(x',x'')\in \Omega\cap B_1$.

Let now $g$ be the harmonic function in $B_{1,+}$ such that $g=h$ on $\partial B_{1,+}\cap\Omega$ and $g=0$ on $\partial B_{1, +}\setminus\Omega$. We have $|g(x)|\le C_1(x''-x_1'')$ in $B_{1,+}$ by the above estimate on $h$ and the maximum principle.
We consider the difference $g-h$. We have $g=h$ on $\Omega\cap\partial B_1$ and $|g-h|=|g|\le 4C_1\tau$ on $\partial\Omega\cap B_1$. Then, by the maximum principle, 
$|g-h|\le 4C_1\tau\quad {\text{in}}\quad \Omega\cap B_1.$ We extend $h$ by zero to $B_{1,+}\setminus \Omega$. Then $|g-h|\le 4C_1\tau$ in $B_{1,+}$.

 Let $m$ be an integer such that $2\sqrt{d}\le 2^m<4\sqrt{d}$, clearly $m\ge 1$. Then the estimate $N^*_h(Q)\le N$ implies $N_h(x_Q,2^m)\le N$. We choose $\varepsilon$ such that $(1+\varepsilon)^{m+3}\le 2$ and assume that $\tau<\tau_\varepsilon$ using the notation of Lemma \ref{l:mon}. Then  $N_h(x_Q,2^{j})\le 2N$ when $-3\le j\le m$.  
 We use \eqref{eq:hnorm} and \eqref{eq:subh} to conclude  that
\[
\int_{ B(x_Q,\frac{1}{8})\cap\Omega}h^2
\ge e^{-10N}\int_{B(x_Q,4)\cap\Omega}h^2\ge ce^{-10N}.
\]
Suppose that $\tau<\frac1{24}$. Then $B(x_Q, \frac18)\cap\Omega\subset \frac{1}{4}B_{1,+}$ and we have
\[
\left(\int_{ \frac{1}{4}B_{1,+}}g^2\right)^{1/2}\ge \left(\int_{\frac{1}{4}B_{1,+}}h^2\right)^{1/2}-C_2\tau\ge  \left(\int_{B(x_Q,\frac{1}{8})\cap\Omega}h^2\right)^{1/2}-C_2\tau.\]
Assuming that $\tau(N)$ is small enough, we conclude that 
\begin{equation}\label{eq:gnorm}
\int_{\frac{1}{4}B_{1,+}} g^2\ge c_1e^{-10N}.
\end{equation} We also have $\sup_{\frac{1}{2}B_{1,+}}|g|\le\sup_{B_1\cap\Omega}|h|\le 1$ by \eqref{eq:hnorm}.
Then
 \[N_g(x_1, \frac{1}{4})=\log\frac{\int_{\frac12B_{1,+}}g^2}{\int_{\frac14B_{1,+}}g^2}\le C(N+1).\]
 
 We note that \eqref{eq:gnorm} implies $\sup_{\frac14B_{1,+}}|g|\ge ce^{-5N}$. Then, by Lemma \ref{l:toyl}, there exist  $x_*\in \Gamma_0\cap \frac{1}{16}B_1$, $c_2=c_2(N)>0$, and $\rho=\rho(C(N+1))$ such that
\[|g(x)|\ge c_2(x''-x_1'')\quad {\text{for}}\quad x=(x',x'')\in B(x_*,\rho)\cap B_{1,+}.\]
We may assume that $g>0$ in $B(x_*,\rho)\cap B_{1,+}$, otherwise we consider $-h$ in place of $h$. Then we obtain 
\[h(x)\ge g(x)-4C\tau \ge c_2(x''-x_1'')-4C\tau\quad{\text{in}}\quad B(x_*,\rho)\cap \Omega.\]

We note that $\rho$ does not depend on $\tau$ and for $\tau$ small enough we have $B(x_*,\frac{\rho}{4})\cap\partial\Omega\neq\varnothing$. We also have $B(x_*,\frac{\rho}{2})\subset Q$. 

Our goal is to show that $h>0$ on $B(x_*,\frac{\rho}{2})\cap\Omega$. Let $y_*=(y_*',y_*'')\in B(x_*,\frac{\rho}{2})\cap\partial\Omega$. We note that
\begin{equation}\label{eq:n1}
h(x)\ge c_2(x''-y_*'')-c_3\tau \quad{\text{in}}\quad B(x_*,\rho)\cap \Omega,
\end{equation}
where  $c_3=4C+4c_2$. We consider the harmonic function 
\[h_*(x)=\frac{1}{d\rho}\left((d-1)(x''-y''_*)^2-|x'-y'_*|^2\right),\]
where $x=(x',x'')$.
We claim that 
$h(x)\ge c_2 h_*(x),$ when $x\in B\left(y_*,\frac{\rho}{2}\right)\cap\Omega$ and
 $\tau$ is small enough.
 
 First, we note that $h_*(x)\le 0$ if $|x''-y_*''|\le(d-1)^{-1/2}|x'-y_*'|$ and therefore $h_*\le 0$ on  
$\partial\Omega\cap B(y_*,\frac{\rho}{2})$ when $\tau$ is small enough, while $h=0$ on $\partial\Omega\cap B(y_*,\frac{\rho}{2})$.
On $\partial B(y_*, \frac{\rho}{2})\cap\Omega$ we have
\[h_*(x)=\frac{(x''-y_*'')^2}{\rho}-\frac{\rho}{4d}.\]
Comparing \eqref{eq:n1} to the last identity and denoting $t=x''-y_*''$, we reduce the inequality $h\ge c_2h_*$ on $\partial B(y_*,\frac{\rho}{2})\cap\Omega$ to the following one:
\[c_2t-c_3\tau\ge c_2\left(\frac{t^2}{\rho}-\frac{\rho}{4d}\right),\]
when $t\in (-\tau\rho/2, \rho/2)$ and $\tau$ is small enough. It sufficies to check the inequality for $t=-\tau\rho/2$ and $t=\rho/2$.  For $t=-\tau\rho/2$ we obtain the inequality 
\[\frac{c_2\rho}{4d}\ge \tau\left(\frac{c_2\rho}{2}+\frac{c_2\tau\rho}{4}+c_3\right),\]
 which holds when $\tau$ is small enough. On the other hand, for $t=\rho/2$, the inequality is  reduced to 
 \[c_2\rho\left(\frac14+\frac1{4d}\right)\ge c_3\tau.\] This one is also satisfied for small $\tau$.

Thus, by the maximum principle,  $h\ge c_2h_*$ in $B(y_*,\rho/2)\cap\Omega$. In particular, $h(y_*',y'')\ge h_*(y_*', y'')>0$ when $y_*''<y''<\rho/2$. 
Therefore $h>0$ on $B(x_*,\frac{\rho}{2})\cap\Omega$. 

Finally, since $B(x_*,\rho/2)$ contains a ball of radius $\rho/4$ centered on $\partial\Omega$, if $k$ is large enough, there is $q\in\mathcal{B}_k(Q)$ such that $q\subset B(x_*,\frac{\rho}{2})$ and then $Z(h)\cap q=\varnothing$.
\end{proof}

\section{Proof of  Theorem \ref{th:m2}}\label{s:main}

Let $N_0$ be as in Lemma \ref{l:hp} and let $\Omega$, $B=B(x,r)$, and $h$ be as in the statement of Theorem \ref{th:m2}. We remind that the maximal doubling index $N_h^*(Q)$ of  $h$ in a cube $Q$ was defined by \eqref{eq:N*}. For the rest of the proof we modify the maximal doubling index and write $N^{**}_h(Q)=\max\{N^*_h(Q),N_0/2\}$. 
  Then Lemmas \ref{l:hp} and \ref{l:hp1} imply that there is $k$ such that for $\tau$ small enough, if  $Q\subset 2B$ and $(Q,\mathcal{B}_k(Q), \mathcal{I}_k(Q))$ is a standard construction, then there is a cube $q_0\in\mathcal{B}_k(Q)$ such that 
 \begin{equation}\label{eq:stepM}
 {\text{either}}\quad (i)\ N_h^{**}(q_0)<N_h^{**}(Q)/2\quad{\text{or}}\quad (ii)\ Z(h)\cap q_0=\varnothing.
 \end{equation}

\subsection{Reduction to one cube}
Let $Q$ be a cube as above. We claim that
 \begin{equation}\label{eq:dc}
 \h^{d-1}(Z(h)\cap Q)\le CN_h^{**}(Q)s(Q)^{d-1}.
 \end{equation}
 
 Assume first that \eqref{eq:dc} holds. We show that Theorem \ref{th:m2} follows.  We need to switch from cubes to balls and from the maximal doubling index to the doubling index at a single point.
 
    To this end, we cover the ball $B(x,r)$ with cubes $Q_j\subset B(x,2r)$ such that $\diam(Q_j)=r/10$ and either  $\dist(Q_j, \partial\Omega)>s(Q_j)/10$ (inner cubes) or $Q_j$ satisfies the assumptions in the main construction (boundary cubes). We  may assume that there are not more than  $C=C(d)$ of such cubes. 
  
   First, for each  cube $Q=Q_j$ in this cover, we have $Q\cap B(x,r)\neq\varnothing$, and we compare $N_h^*(Q)$ to $N_h(x,4r)$. 
   There exists $y\in Q\cap\overline{\Omega}$  and $r_y\in[r/20, r/10]$ such that $N_h^*(Q)=N_h(y,r_y)$. Assuming that $\tau<\tau_1$ in the notation of Lemma \ref{l:mon}, we get $N_h(y,32r_y)\ge 2^{-5}N_h^*(Q)$.   We have $\dist(x, y)\le \frac{11}{10}r$ and 
 \[N_h(y, 32 r_y)=\log\frac{\int_{B(y, 64 r_y)\cap\Omega}h^2}{\int_{B(y, 32 r_y)\cap\Omega}h^2}\le \log\frac{\int_{B(x, 8r)\cap\Omega}h^2}{\int_{B(x,r/2)\cap\Omega}h^2}\le 16 N_h(x,4r)\]
 by Lemma \ref{l:mon}. Hence, $N^*_h(Q)\le 2^9 N_h(x,4r)$ and $N^{**}_h(Q)\le C(N_h(x,4r)+1)$.
 
 Each inner cube  $Q\subset\Omega$ can be covered by at most $C$ balls $b$  with centers in $Q$ and with radii  $s(Q)/100$. Then $8\overline{b}\subset\Omega$. Moreover,  if $b=B(y, s(Q)/100)$, we have  $N_h(y,s(Q)/25)\le CN_h^{*}(Q)$ by Lemma \ref{l:mon} again.  Then we use  Lemma \ref{l:in} to estimate the area of the zero set of $h$ in each of the balls $b$ and obtain
 \begin{multline*}
 \h^{d-1}(Z(h)\cap b)\le C(N_h(y, s(Q)/25)+1)r^{d-1}\\
 \le C'(N_h^*(Q)+1)r^{d-1}\le C''(N_h(x,4r)+1)r^{d-1}.
 \end{multline*}
  
 For the boundary cubes, we use the  inequality \eqref{eq:dc}. Thus for every $Q_j$, we obtain
 \[\h^{d-1}(Z(h)\cap Q_j)\le C(N_h(x,4r)+1)s(Q_j)^{d-1}.\]
 Summing these inequalities over all cubes, we obtain the required estimate. It remains to prove \eqref{eq:dc}. 

\subsection{Proof of \eqref{eq:dc}}\label{ss:ptm2}
We fix a compact set $K\subset\Omega$ and prove that 
\begin{equation}\label{eq:dcK}
\h^{d-1}(Z(h)\cap Q\cap K)\le C_0N_h^{**}(Q)s(Q)^{d-1},
\end{equation}
where $Q\subset 2B$ is a cube as in the standard construction and $C_0$ is independent of $K$. Then \eqref{eq:dc} follows.

First, note that \eqref{eq:dcK} holds for all cubes $Q$ small enough, since $Q\cap K=\varnothing$ for such cubes. We prove \eqref{eq:dcK} by induction on the size of $Q$, going from small cubes to larger ones. 
Assume that it holds for cubes with $s(Q)<s$, we want to prove it for cubes with $s(Q)<2^ks$, where $k$ is as in \eqref{eq:stepM}.

We consider the standard construction $(Q, \mathcal{B}_k(Q), \mathcal{I}_k(Q))$. 
Each inner cube $q\in\mathcal{I}_k(Q)$ can be covered  by balls $b$  centered in $q$ with radii  $s(q)/100$ and such that $8\overline{b}\subset\Omega$, so that the number of balls is bounded by a dimensional constant. For each such ball $b=B(y, s(q)/100)$, applying Lemma~\ref{l:mon}, we get $N_h(y,s(q)/25)\le C(k)N^*_h(Q)$ when $\tau$ is small enough. Then by  Lemma \ref{l:in},  we have 
\begin{equation}\label{eq:innerc}
\sum_{q\in\mathcal{I}_k(Q)}\h^{d-1}(Z(h)\cap q)\le C (N_h^*(Q)+1)s(Q)^{d-1}\le C_1N_h^{**}(Q)s(Q)^{d-1},
\end{equation}
where $C$ and $C_1$ depend on $k$.

For all other boundary cubes $q$, we have $N^{**}_h(q)\le (1+\varepsilon)^kN^{**}_h(Q)$. Also \eqref{eq:stepM} implies that there is a cube $q_0\in\mathcal{B}_k(Q)$ such that either $N^{**}_h(q_0)\le N_h^{**}(Q)/2$ or $Z(h)\cap q_0=\varnothing$. 
We apply the induction assumption to each boundary cube
and obtain
\begin{multline*}
\h^{d-1}(Z(h)\cap K\cap(\cup_{q\in\mathcal{B}_k(Q)}q))\\
\le \sum_{q\in\mathcal{B}_k(Q),q\neq q_0}\h^{d-1}(Z(h)\cap K\cap  q)+\h^{d-1}(Z(h)\cap K\cap q_0)\\
\le \sum_{q\in\mathcal{B}_k(Q),q\neq q_0}C_0 N_h^{**}(q)s(q)^{d-1}+\frac {C_0}{2} N_h^{**}(Q)s(q_0)^{d-1}\\
\le
\left(\frac{2^{k(d-1)}-1}{2^{k(d-1)}}(1+\varepsilon)^k+\frac{1}{2}\cdot\frac{1}{2^{k(d-1)}}\right) C_0N_h^{**}(Q)s(Q)^{d-1}.\end{multline*}
Finally, we choose $\varepsilon$ small and $C_0$ large enough so that
\[C_1+\left(\frac{2^{k(d-1)}-1}{2^{k(d-1)}}(1+\varepsilon)^k+\frac{1}{2}\cdot\frac{1}{2^{k(d-1)}}\right) C_0<C_0.\]
Note that $C_0$ does not depend on $K$.
Then, assuming that $\tau$ is small enough and taking into account \eqref{eq:innerc}, we obtain
\[
\h^{d-1}(Z(h)\cap K\cap Q)\le 
C_0N_h^{**}(Q)s(Q)^{d-1}.
\]
This concludes the induction step and the proof of \eqref{eq:dc}.


\section{Dirichlet Laplace eigenfunctions}\label{s:eigen}

\subsection{Harmonic extension and an estimate of the doubling index}\label{ss:harm}
 Let $\Omega_0\subset \R^n$ be a bounded Lipschitz domain. Let $u_\lambda$ be an eigenfunction of the Dirichlet Laplace operator, $u_\lambda\in W_0^{1,2}(\Omega_0),$ $\Delta u_\lambda+\lambda u_\lambda=0$. Then $u_\lambda\in C(\overline{\Omega}_0)$.  This fact is well-known, we provide a proof in the Appendix below, see Section \ref{s:Hol}.
 
 We consider the harmonic extension of $u_\lambda$ to the domain $\Omega=\Omega_0\times\R\subset \R^{n+1}$, given by
\[h(x,t)=u_\lambda(x)e^{\sqrt{\lambda}t}.\]
Then $h\in C(\overline{\Omega})$ and, clearly, $Z(h)=Z(u_\lambda)\times\R$, where the zero sets are sets inside the domains $\Omega$ and $\Omega_0$ respectively. We need the following estimate of the doubling index of this harmonic extension.

\begin{lemma}
\label{l:DF}
Let $\Omega_0$ be a bounded domain in $\R^n$ with a sufficiently small local Lipschitz constant $\tau$.  Let $r_0>0$ be such that $\partial\Omega_0\cap B(x,r_0)\in Lip(\tau)$ for any $x\in\partial\Omega_0$.  Then for any $r\in (0, r_0/16)$, there exists $C=C(r,\Omega_0)>0$   such that for any Dirichlet Laplace eigenfunction $u_\lambda$, the corresponding harmonic extension $h(x,t)=u_\lambda(x)e^{\sqrt{\lambda}t}$ satisfies   $N_h(y,r)\le C\sqrt{\lambda}$ when $y=(x,t)\in\overline{\Omega}$.
\end{lemma}

This result is similar to the results of Donnelly and Fefferman, \cite{DF, DF1}, who  considered eigenfunctions on compact manifolds and on domains with $C^\infty$-smooth boundaries and obtained the above estimate of the doubling index for eigenfunctions. However, in contrast to the previous results,  the doubling index is allowed to blow up as $r\to 0$ in the above lemma.  The statement of the lemma follows by  application of  Lemma \ref{l:three1} and inequality \eqref{eq:threeh} to a chain of balls, the argument is similar to the one in  \cite[Section 2.4]{LM}. For the convenience of the reader, we provide the details below.

\begin{proof} 
 We consider  any  $y=(x,t)\in\overline\Omega$ and let $y_0=(x,0)$. Since $h(x,t+s)=e^{\sqrt{\lambda}t}h(x,s)$, we have $N_h(y,r)=N_h(y_0,r)$. So it is enough to estimate the doubling index of $h$ in the balls centered on $\overline{\Omega}_0\times\{0\}$. 
 
 We fix $r\in(0, r_0/16)$ and let $\mathcal{S}\in\overline{\Omega}_0$ be a finite $r/8$-net for $\overline{\Omega}_0$, i.e., $\overline{\Omega}_0\subset\bigcup_{p\in \mathcal{S}} B(p,r/8)$.
 Let $B_*=B(y,r)$ be  a ball of radius $r$ centered at $y=(y_*,0)\in\overline{\Omega}_0\times\{0\}$.
 Assume that $\max_{\Omega_0}|u_\lambda|=|u_\lambda(x_0)|=1$. We consider a path $\gamma:[0,1]\to\overline{\Omega}_0$ from $y_*$ to $x_0$ such that $\gamma((0,1))\subset\Omega_0$. Now we construct a chain of balls $\{B_j\}_{j=0}^J$. Let $B_0=B(y_*,r/2)$. Assuming that $B_j=B(y_j,r/2)$ is constructed, we define 
 \[s_j=\sup\{s\in[0,1]: |\gamma(s)-y_j|\le r/8\}.\]
 If $s_j<1$, we have $|\gamma(s_j)-y_j|=r/8$ and we choose $y_{j+1}\in \mathcal{S}$ such that $|y_{j+1}-\gamma(s_j)|<r/8$. If $s_j=1$, we define $y_{j+1}=y_J=x_0$ and stop the chain. We have $|y_j-y_{j+1}|< r/4$ and define $B_{j+1}=B(y_{j+1}, r/2)$. We note that $s_{j+1}>s_j$ when $0\le j<J-1$ and that $y_{j+1}\in\mathcal{S}\setminus\{y_0,...,y_j\}$ when $0\le j< J-1$. We also have $B_{j+1}\subset \frac{3}{2}B_j$. 
  The resulting chain is finite, moreover, the number of balls in the chain is bounded by the number of elements in $\mathcal{S}$ plus two.

Let now $\widetilde{B}_j=B((y_j,0),r/2)$ be the corresponding ball in $\R^{n+1}$. 
Then $\sup_{4\widetilde{B}_j\cap\Omega}|h|\le e^{2\sqrt{\lambda} r}$. If $4\tilde{B}_j\subset\Omega$, then \eqref{eq:threeh} gives
\begin{multline*}\sup_{\frac{3}{2}\widetilde{B}_j}|h|\le 2^{n+1}(\sup_{\widetilde{B}_j}|h|)^{1/2}(\sup_{4\widetilde{B}_j}|h|)^{1/2}\\
\le  3^{n+1}(\sup_{\widetilde{B}_j}|h|)^{1/3}(\sup_{4\widetilde{B}_j}|h|)^{2/3}\le 3^{n+1}e^{4\sqrt{\lambda}r/3}(\sup_{\tilde{B}_j}|h|)^{1/3}.\end{multline*}
Otherwise we have $\dist(y_j,\partial\Omega_0)<2r<r_0/8$. In this case, there is a ball $\tilde{B}$ of radius $r_0$ centered on $\partial\Omega_0\times\{0\}$  such that $(y_j,0)\in \overline{\Omega}\cap \frac{1}{4}\tilde{B}$ and $16 \tilde{B}_j\subset \tilde{B}$. Then  Lemma \ref{l:three1}, applied to the ball $\tilde{B}_j$, implies that 
\[\sup_{\frac32 \tilde{B}_j\cap\Omega}|h|\le 3^{n+1}(\sup_{\tilde{B}_j\cap\Omega}|h|)^{1/3}(\sup_{4\tilde{B}_j\cap\Omega}|h|)^{2/3}\le 3^{n+1}e^{4\sqrt{\lambda}r/3}(\sup_{\tilde{B}_j\cap\Omega}|h|)^{1/3}.\]
Therefore, we obtain for each $j$, 
\[\sup_{\widetilde{B}_{j}\cap\Omega}|h|\ge 3^{-3(n+1)}(\sup_{\frac32\widetilde{B}_{j}\cap\Omega}|h|)^3 e^{-4\sqrt{\lambda} r}\ge 3^{-3(n+1)}(\sup_{\widetilde{B}_{j+1}\cap\Omega}|h|)^{3}e^{-4\sqrt{\lambda} r}.\]
We also have $\sup_{\widetilde{B}_J\cap\Omega}|h|=e^{\sqrt{\lambda} r/2}$. Combining the above inequalities, we get
\[\sup_{\widetilde{B}_0\cap\Omega}|h|\ge c_1e^{-C_2\sqrt{\lambda}},\]
 where $c_1$ and $C_2$ depend on $r$ and $J$ but not on $\lambda$. We can choose the $r/8$-net $\mathcal{S}$ so that the number of points in $\mathcal{S}$ depends only on $\diam(\Omega_0)$, $r$, and the dimension. Thus we conclude that the constants in the last inequality depend only on $r$, the diameter of $\Omega_0$, and $n$. 
 
Finally, applying \eqref{eq:subh}, we obtain
\[N_h(y,r)=\log\frac{\int_{4\widetilde{B}_0\cap\Omega}h^2}{\int_{2\widetilde{B}_0\cap\Omega}h^2}\le \log\frac{\sup_{4\widetilde{B}_0\cap\Omega}|h|^2}{\sup_{\widetilde{B}_0\cap\Omega}|h|^2}+C\le (4r+2C_2)\sqrt{\lambda}+C\le C\sqrt{\lambda},\]
where $C=C(\Omega_0,r)$. We remark that $\lambda\ge\lambda_1(\Omega_0)>0$, where $\lambda_1(\Omega_0)$ is the first Dirichlet Laplace eigenvalue in $\Omega_0$. Moreover, if $B^*$ is a ball of radius $\diam(\Omega_0)$ then $\lambda_1(\Omega_0)\ge \lambda_1(B^*)$. Thus the constant $C$ in the conclusion of this Lemma depends only on $r$, $\diam(\Omega_0)$, and $n$.
\end{proof}

\subsection{Proof of Theorem $\ref{th:main}'$}\label{ss:thm}
Let $\Omega_0\subset\R^n$ be a bounded  domain with a sufficiently small local Lipschitz constant $\tau$.  Let also $r_0>0$ be such that $\partial\Omega_0\cap B(x,r_0)\in Lip(\tau)$ for every $x\in\partial\Omega_0$. We consider the domain $\Omega=\Omega_0\times\R\subset \R^{n+1}$ and let $\Omega_1=\Omega_0\times[-1,1]$. For each $x\in\partial\Omega\times[-1,1]$ we consider a ball centered at $x$ of radius $2^{-9}r_0$. These balls cover the closed $2^{-10}r_0$-neighborhood of the set  $\partial \Omega\times[-1,1]$. We can choose a disjoint collection of these balls $b_j$ such that the balls $B_j=4b_j$ cover the same closed neighborhood of  $\partial \Omega\times[-1,1]$.  Then for each point  of $\Omega_1\setminus\cup_j B_j$, we choose a ball $b$ centered at the point of radius $2^{-15}r_0$, so that $32b\subset \Omega$. Once again, we find a finite sub-collection of  disjoint balls $b'_k$ such that   $B_k'=4b_k'$ cover $\Omega_1\setminus\cup_j B_j$. We note that $8B_k'\subset \Omega$.  We fix this covering  of $\Omega_1$ and remark that radii of all balls depend only on $r_0$ and the number of balls depends on $r_0$, the diameter of $\Omega_0$, and $n$.

Let now $u_\lambda$ be a Dirichlet Laplace  eigenfunction in $\Omega_0$: $\Delta u_\lambda+\lambda u_\lambda=0$ in $\Omega_0$ and $u_\lambda=0$ on $\partial\Omega_0$. We consider its harmonic extension $h(x,t)=e^{\sqrt{\lambda}t}u_\lambda(x)$. Then $h\in C(\overline{\Omega})$ is  non-zero,  and $h=0$ on $\partial\Omega$. Let $C_0=\max\{C(2^{-5}r_0,\Omega_0), C(2^{-11}r_0,\Omega_0)\}$, where $C(r,\Omega_0)$ is as in  Lemma \ref{l:DF}. Then for $B(x,r)\in\{B_j\}\cup\{B_k'\}$, we have $N_h(x,4r)\le C_0\sqrt{\lambda}$.
Finally, we apply Theorem \ref{th:m2} to each of the balls $B_j$ and Lemma \ref{l:in} to each of the inner balls $B_k'$. We conclude that
\begin{multline*}
\h^{n}(Z(h)\cap\Omega_1)\le\sum_j\h^n(Z(h)\cap B_j)+\sum_k\h^n(Z(h)\cap B'_k)\\
\le  C(C_0\sqrt{\lambda}+1)\left(\sum_j r(B_j)^n+\sum_kr(B'_k)^n\right)\le C_1\sqrt{\lambda}.
\end{multline*}
 Then $\h^{n-1}(Z(u_\lambda)\cap\Omega_0)\le C_1\sqrt{\lambda}$, which finishes the proof of Theorem $\ref{th:main}'$.

\appendix

\section*{Appendix: Proofs of some auxiliary results}
\renewcommand{\thesubsection} {A.\arabic{subsection}}

\subsection{Estimates for the zero set of harmonic functions inside the domain}\label{ss:DF} We outline some steps of the proof of Lemma \ref{l:in}. First the harmonic function $h$ is extended to a holomorphic function $H$ on a domain in \ $\C^{d}$,
see Lemma 7.2 in \cite{DF}. Our situation is particularly simple, since we only consider the standard Laplace operator on Euclidean domains. For this case  the  holomorphic extension is given by the complexification of the Poisson kernel.
The Poisson kernel in a ball $B(x,r)\subset\R^d$ is given by
\[P_r(z,y)=c_d\frac{r^2-|z-x|^2}{r|z-y|^{d}},\quad |z-x|<r,\ |y-x|=r.\]
For any $y\in\partial B(x,r)$,
the function $z=(z_1,...,z_d)\mapsto \sum_j(z_j-y_j)^2$ maps the complex ball $B_\C(x,r/\sqrt{2})\subset \C^d$  of radius $r/\sqrt{2}$ centered at $x\in\R^d\subset\C^d$ to the half-plane $\Re \xi>0$. Then the Poisson kernel has the holomorphic extension to $B_\C(x,r/\sqrt{2})$. Moreover, for any $a<1/\sqrt{2}$, 
\[|P_r(z,y)|\le C(a)r^{-(d-1)},\quad z\in B_\C(x,r_0),\  r_0\le ar.\]

We consider a  ball $B=B(x,8r)$ such that $\overline{B}\subset\Omega$. Then there exists a holomorphic extension $H(z)$ of $h$ defined on a ball $B_{\C}(x, 3r)$,
\[H(z)=\int_{\partial B(x,6r)}P_{6r}(z,y)h(y)d\sigma(y),\]
such that $|H(z)|\le C\max_{\overline{B}(x, 6r)}|h|$. Then
\[\sup_{B_{\C}(x,3r)}|H(z)|\le C'r^{-d/2}\left(\int_{B(x,8r)}h^2\right)^{1/2}.\]  

Now we can cover the set $Z(h)\cap B(x,r)$ by a finite number of balls with centers in $B(x,r)$ of radii $r/20$ so that the number of the balls is bounded by a constant depending on the dimension only. Let $B(y,r/20)$ be one of such balls. By a version of Corollary \ref{cor:x12} for the doubling index inside the domain, we have
$N_h(y, 2r)\le 3N$, where $N=N_h(x,4r),$ and, therefore, $N_h(y, r_1)\le 3N$ when $r_1<2r$.
Thus 
\[\sup_{B(y,\frac{r}{16})}h^2\ge cr^{-d}\int_{B(y,\frac{r}{16})}h^2 \ge cr^{-d}e^{-15N}\int_{B(y,2r)}h^2\ge cr^{-d}e^{-15N}\int_{B(x,r)}h^2.\] 
Therefore,
\[\sup_{B(y,\frac{r}{10})}|H|\ge \sup_{B(y,\frac{r}{16} )}|h|\ge cr^{-d/2}e^{-7.5N}\left(\int_{B(x,r)}h^2\right)^{1/2}.\]
Combining the inequalities above, we obtain
\[\frac{\sup_{B_{\C}(y,2r)}|H|}{\sup_{B(y,r/10)}|H|}\le \frac{\sup_{B_{\C}(x,3r)}|H|}{\sup_{B(y,r/10)}|H|}\le Ce^{7.5N}\left(\frac{\int_{B(x,8r)}h^2}{\int_{B(x,r)}h^2}\right)^{1/2}\le Ce^{9N}.\]
Finally an estimate for the size of the zero set of a holomorphic function, Proposition 6.7 in \cite{DF}, implies that
\[\h^{d-1}(Z(h)\cap B(y, r/20))\le C(N_h(x,4r)+1).\]
We sum these inequalities over all balls $B(y, r/20)$ to obtain the required estimate for $\h^{d-1}(Z(h)\cap B(x,r))$.

\subsection{Continuity of eigenfunctions in Lipschitz domains}\label{s:Hol}
  First we prove the following regularity result.
 \begin{lemma}\label{l:brn}
 Let $\Omega$ be a domain in $\R^d$  and let $h$ be a harmonic function in $\Omega$. Suppose that $B$ is a ball  centered on $\partial\Omega$ and that there exists a sequence of functions $\{h_n\}$, $h_n\in C^\infty_0(\R^d)$ with the support of $h_n$ contained in $\Omega$, such that $h_n\to h$ and $\nabla h_n\to \nabla h$ in $L^2(B\cap\Omega)$. Assume also that $\partial\Omega\cap B\in Lip(\tau)$ and define $h=0$ on $\partial\Omega\cap B$. Then $h\in C(\overline{\Omega}\cap\frac12B)$.
 \end{lemma}

\begin{proof} 
We define the function
\[v=\begin{cases} h^2\ {\text{in}}\ \Omega\cap B,\\ 0\ {\text{in}}\ B\setminus \Omega.\end{cases}\]
Then $v\in L^1(B)$.   Let $\varphi\in C_0^\infty(B)$. We have
\begin{multline}\label{eq:app1}
\int_B v\Delta \varphi=\lim_{n\to\infty} \int_B h_n^2\Delta \varphi\\=
-2\lim_{n\to\infty} \int_B h_n\nabla h_n\cdot\nabla \varphi=-2\int_{B\cap\Omega} h\nabla h\cdot\nabla\varphi.
\end{multline}
On the other hand, since $h$ is harmonic in $\Omega$ , we obtain
\[0=\int_{\Omega} \nabla h\cdot \nabla(h_n\varphi)=\int_\Omega h_n\nabla h\cdot\nabla\varphi+\int_\Omega \varphi\nabla h\cdot\nabla h_n.\]
Taking the limit as $n\to\infty$, we get
\[ \int_\Omega h\nabla h\cdot\nabla\varphi=-\int_{\Omega} \varphi |\nabla h|^2.\]
Combining the last  identity and \eqref{eq:app1} gives
\[
\int_B v\Delta\varphi=2\int_{B\cap\Omega}|\nabla h|^2\varphi.\]
In particular, $v$ is subharmonic in $B$ in the weak sense: If $\varphi\ge 0$, $\varphi\in C_0^\infty(B)$, then $\int_B v\Delta\varphi\ge 0$.  If $\alpha$ is a standard mollifier, $\alpha_\delta(x)=\delta^{-d}\alpha(\delta^{-1}x)$, and $v_\varepsilon=v*\alpha_{\varepsilon r}$, where $r$ is the radius of $B$. Then $v_\varepsilon$ is subharmonic in $(1-\varepsilon)B$ and $v_\varepsilon\to v$ in $L^1(B)$ and almost everywhere. In particular, $v$ satisfies the mean value inequality at each its Lebesgue point. Clearly any $y\in\Omega\cap B$ is a Lebesgue point of $v$ as $v=h^2$ in $\Omega\cap B$ and $h\in C(\Omega)$. So for any $y\in\Omega\cap B$ and any ball $B_1\subset B$ centered at $y$ we have
\[v(y)\le \frac{1}{|B_1|}\int_{B_1}v.\] 
In particular, 
\[\sup_{\frac23B\cap\Omega} h^2\le \frac{3^d}{|B|}\int_{B\cap\Omega}h^2<\infty.\]

 Suppose that
 $x_1\in\partial\Omega\cap \frac12 B$. There exists a cone $\mathcal{C}$ with the vertex at $x_1$ such that ${\mathcal C}\cap(\Omega\cap B)=\varnothing$ and the aperture of $\mathcal{C}$ does not depend on $x_1$ (it depends on $\tau$ only). We use the following simple fact. If $y_1\in \R^d$ and $\rho>2\,\dist(x_1,y_1)$, then 
 \[| B(y_1, \rho)\cap\mathcal{C}|\ge \alpha| B(y_1,\rho)|,\] for some $\alpha=\alpha(\tau)\in(0,1)$.
 
 Let  $m_k=\sup_{B(x_1, 3^{-k}r)\cap\Omega}|h|$ for $k\ge 2$. We know that $m_k<\infty$. Let $y\in B(x_1, 3^{-k}r)\cap\Omega$, $k\ge 3$. By the mean value theorem applied to $v$, we obtain
 \[v(y)\le \frac{1}{|B(y, 2\cdot 3^{-k}r)|}\int_{B(y, 2\cdot 3^{-k}r)} v\le (1-\alpha)m_{k-1}^2.\]
 Thus $\sup_{B(x_1, 3^{-k}r)\cap\Omega}|h|\le (1-\alpha)^{(k-2)/2}\sup_{\frac{2}{3}B\cap\Omega}|h|$. We conclude that
\[\lim_{y\to x_1, y\in\Omega} h(y)=0.\]
 \end{proof}
 We remark that the argument above implies that $h$ is H\"older continuous in $\overline\Omega\cap B$ and there exist $C>0$ and $\beta\in(0,1)$ such that
 \[|h(y)|\le C\dist(y,\partial\Omega)^\beta r^{-\beta}\sup_{\Omega\cap \frac{2}{3}B}|h|, \quad y\in \Omega\cap\frac{1}{2}B.\]
  
 \begin{corollary} Let $\Omega_0\subset\R^n$ be a bounded Lipschitz domain. Let $u_\lambda$ be Laplace Dirichlet eigenfunction in $\Omega_0$. Then $u_\lambda$ extended by zero to $\partial\Omega_0$ is continuous on $\overline{\Omega}_0$.
 \end{corollary}
 
 \begin{proof} We have $u_\lambda\in W^{1,2}_0(\Omega_0)\cap C^\infty(\Omega_0)$ and $\Delta u_\lambda+\lambda u_\lambda=0$ in $\Omega_0$. We consider the harmonic function $h(x,t)=e^{\sqrt{\lambda}t}u_\lambda(x)$ in $\Omega=\Omega_0\times\R$. We note that for any $B$ centered on $\partial\Omega$, $h$ satisfies the assumptions of Lemma \ref{l:brn}. Then $h$ is continuous in $\overline{\Omega}$ and vanishes on $\partial\Omega$. This implies that $u_\lambda\in C(\overline{\Omega}_0)$ and $u_\lambda=0$ on $\partial\Omega_0$.
 \end{proof}

\subsection{Quantitative Cauchy uniqueness}\label{ss:cau}
 We give an elementary proof of  Lemma \ref{l:Cauchy} in this section for the convenience of the reader. 
 
 Let $G(x,y)=-c_d|x-y|^{2-d}$ be the fundamental solution of the Laplace equation in $\R^d$ when $d\ge 3$ (similar computations can be done with $G(x,y)=c_2\log|x-y|$ for $d=2$).
We write $\partial B_+=\Gamma\cup\Sigma$, where $\Gamma$ is the flat part of the boundary and $\Sigma=\partial B_+\setminus\Gamma$. We denote by $n$ the outer normal to $\partial B_+$. Then for $x\in B_+$, the Green formula implies
\begin{multline*}h(x)=\int_{\partial B_+} \left[\frac{\partial G}{\partial n}(x,y)h(y)-G(x,y)\frac{\partial h}{\partial n}(y)\right]dy\\=
\int_\Gamma  \left[\frac{\partial G}{\partial n}(x,y)h(y)
-G(x,y)\frac{\partial h}{\partial n}(y)\right]dy\\+\int_{\Sigma} \left[ \frac{\partial G}{\partial n}(x,y)h(y)-G(x,y)\frac{\partial h}{\partial n}(y)\right]dy\\=h_1(x)+h_2(x).\end{multline*}
 The functions $h_1$ and $h_2$ are defined in the complements of $\Gamma$ and $\Sigma$ respectively and are harmonic in the corresponding domains. Moreover, for $x\not\in\overline{B}_+$, applying the Green formula to the functions $h$ and $G(x,\cdot)$ in $B_+$, we obtain $h_1(x)+h_2(x)=0$. 

First, we estimate the value of  $h_1$ at some point $x=(x',x'')\in B\setminus\Gamma \subset\R^{d-1}\times\R$. We  divide the integral into two
\[h_1(x)=\int_{\Gamma} \frac{\partial G}{\partial n}(x,y)h(y)dy-\int_{\Gamma}G(x,y)\frac{\partial h}{\partial n}(y)dy=I_1(x)+I_2(x).\]
Since $|\partial h/\partial n|<\varepsilon$ on $\Gamma$, the second integral is bounded by
\[|I_2(x)|\le c_d\varepsilon \int_{B^{d-1}(x',2)}|x'-y'|^{2-d}dy'\le C\varepsilon.\]
To estimate the first term, we note that for $y\in\Gamma$,
\[\frac{\partial G}{\partial n}(x,y)=c_d'x''|x-y|^{-d},\]
and thereby
\[\int_{\Gamma} \left|\frac{\partial G}{\partial n}(x,y) \right|dy\le c_d'\int_{\R^{d-1}}\frac{|x''|}{(x''^2+|x'-y'|^2)^{d/2}}dy'=c''_d.\] 

Using that $|h(y)|<\varepsilon$ on $\Gamma$, we conclude that $|I_1(x)|<C\varepsilon$ in $B\setminus\Gamma$. Therefore $|h_1(x)|\le C\varepsilon$ in $B\setminus\Gamma$. Since $h_1(x)+h_2(x)=0$ when $x\in \R^d\setminus \overline{B}_+$, and $|h_1+h_2|=|h|
\le 1$ in $B_+$, 
we obtain that $h_2(x)$ satisfies 
\[|h_2(x)|<C\varepsilon\  \ {\text{in}}\ B_-=B\setminus \overline{B}_+\quad{\text{and}}\quad |h_2(x)|\le 1+C\varepsilon\ \ \text{in}\ B_+.\]
Now we apply  the three sphere inequality \eqref{eq:threeh}. We note that $h_2$ is harmonic in $B$. First we take $x=(0,-1/5)$ and $r=1/5$ and obtain
\[\sup_{B(0,1/10)}|h_2|\le\sup_{B(x,3/10)}|h_2|\le 2^d(\sup_{B(x,1/5)}|h_2|)^{1/2}(\sup_{B(x, 4/5)}|h_2|)^{1/2}\le C\varepsilon^{1/2}.\]
Then we apply the same theorem to the balls centered at the origin three times. Since  $(3/2)^3\cdot(1/10)>1/3$ and $(3/2)^2\cdot(1/10)<1/4$, we conclude that 
\[\sup_{\frac{1}{3}B}|h_2(x)|\le C\varepsilon^{1/16}.\]  Finally, combining the last inequality with the bound  $|h_1|\le C\varepsilon$ in $B_+$, we get the required estimate $|h|\le C\varepsilon^\gamma$ in $\frac13B_+$.

\end{document}